\documentclass[11pt,letterpaper]{amsart}
\usepackage{fancyhdr}
\usepackage{bm}
\usepackage{hyperref}
\usepackage{graphicx} 
\usepackage{nicematrix}
\usepackage{amsthm,amssymb,amsmath}
\usepackage[T1]{fontenc}
\usepackage[utf8]{inputenc}
\usepackage{hyperref}
\usepackage{lipsum}
\usepackage{mathrsfs}
\usepackage{enumitem}
\usepackage{stmaryrd}
\usepackage{setspace}
\usepackage{yfonts}
\usepackage{float}
\usepackage{mathtools}
\usepackage{parskip}
\usepackage[all,cmtip]{xy}
\usepackage{tikz-cd}
\tikzcdset{row sep/normal=50pt, column sep/normal=50pt}

\newtheorem{lemma}{Lemma}[section]
\newtheorem{remark}[lemma]{Remark}
\newtheorem{theorem}[lemma]{Theorem}
\newtheorem{theorem*}{Theorem}

\newtheorem{example*}[lemma]{Example}

\newtheorem{manualtheoreminner}{Theorem}
\newenvironment{manualtheorem}[1]{%
  \renewcommand\themanualtheoreminner{#1}%
  \manualtheoreminner
}{\endmanualtheoreminner}

\usepackage{blindtext}

\usepackage[backend=biber, style=numeric, sorting=nyt]{biblatex}
\title{Extension of Ulrich bundles}
\author{\small{Supravat Sarkar}}
\date{}
\begin{document}

\begin{abstract}
We study extension of Ulrich bundles from a smooth nondegenerate subvariety $X$ of $\mathbb{P}^n$. If $X$ is a complete intersection of dimension $\geq 2$, we show that the extension is not possible except in the trivial case. For an arbitrary $X$, we characterize when the extension is possible, assuming some condition on the extended vector bundle. As an application, we generalize previous results of L{\'o}pez and Zamora. We also give several classes of examples of Ulrich bundles on curves that extend to the ambient projective space.
\end{abstract}
\maketitle
\begin{center}
\textbf{Keywords}: Ulrich bundle, extension
\end{center}
\begin{center}
\textbf{MSC Number:14J60, 14H60} 
\end{center}

\section{Introduction}
We work throughout over the field $k=\mathbb{C}$ of complex numbers. Ulrich bundles are important objects of study from the viewpoints of both commutative algebra and algebraic geometry. Let us recall the definition. Let $X$ be a smooth subvariety of $\mathbb{P}^n$, and $E$ a nonzero vector bundle on $X$. In \cite{eisenbud2003resultants}, the following were proved to be equivalent:
\begin{enumerate}
\item There exists a linear resolution
\[
0 \to L_c \to L_{c-1} \to \cdots \to L_0 \to E \to 0,
\]
with $c=\operatorname{codim}(X,\mathbb{P}^n)$
and $L_i=\mathcal{O}_{\mathbb{P}^n}(-i)^{\oplus b_i}.$

\item The cohomology $H^\bullet(X,E(-p))$
vanishes for $1\le p\le \dim(X).$

\item If $\pi:X\to \mathbb{P}^{\dim(X)}$
is a finite linear projection, then the vector bundle $\pi_*E$
is trivial.
\end{enumerate}
We call $E$ an \textit{Ulrich bundle} if these equivalent conditions are satisfied. By condition $(2)$, Ulrichness of a vector bundle on $X$ depends only on the very ample line bundle $\mathcal{O}_X(1)$, not on the space of sections giving the projective embedding. So for any smooth projective variety $X$ and very ample line bundle $L$ on $X$, we call a vector bundle $E$ on $X$ to be $L$-Ulrich if $H^i(X, E\otimes L^{-p})=0$ for all $i$ and $1\leq p\leq \dim X.$ We usually omit $L$ when $L=\mathcal{O}_X(1)$ for a given projective embedding of $X$. For details about Ulrich bundles, see \cite{beauville2018introduction}, \cite{coskun2017survey} or \cite{costa2021ulrich}. Ulrich bundles are closely related to several important areas: minimal resolution conjecture (\cite{aprodu2012minimal}), representations of Clifford algebras(\cite{coskun2012pfaffian}), Boij--S{\"o}derberg theory (\cite{eisenbud2011boij}).

The goal of this article is to study when an Ulrich bundle on $X$ extends to a vector bundle on the linear span of $X$ in the ambient space $\mathbb{P}^n$. To avoid unnecessarily complicated notations and trivialities, we will assume $X$ is nondegenerate, that is, not contained in a hyperplane. A major conjecture in the area of Ulrich bundles, proposed in \cite{eisenbud2003resultants}, says that every smooth subvariety of $\mathbb{P}^n$ carries an Ulrich bundle. This was proved for complete intersections in \cite{herzog1991linear}. Our first result says that a Ulrich bundle on a complete intersection of dimension $\geq 2$ never extends as a vector bundle to the ambient space, except in the trivial cases.
\begin{manualtheorem}{A}\label{A}
    Let $X$ be a smooth nondegenerate subvariety of $\mathbb{P}^n$
of dimension $\ge 2$ which is a complete intersection.
Suppose $\mathcal{E}$ is a vector bundle on $\mathbb{P}^n$ such that
$\mathcal{E}|_X$ is Ulrich. Then $X=\mathbb{P}^n$
and $\mathcal{E}$ is trivial.
\end{manualtheorem}
Next, we study extension of Ulrich bundles from an arbitrary smooth nondegenerate subvariety of $\mathbb{P}^n$, but we impose some condition on the extended vector bundle. For $n\geq 1$ and a nonzero vector bundle $\mathcal{E}$ on $\mathbb{P}^n$, let $c_1(\mathcal{E})$ be its first Chern class, regarded as an integer by the identification $\mathbb{Z}\cong \text{Pic }\mathbb{P}^n$ taking $1$ to the ample generator. Define $\mu(\mathcal{E}):=\frac{c_1(\mathcal{E})}{\operatorname{rank}\mathcal{E}}$, and $\tau(\mathcal{E})=\mu(\mathcal{E}(k))$, where $k$ is the smallest integer such that $\mathcal{E}(k)$ is globally generated. One clearly have  $\tau(\mathcal E(t))=\tau(\mathcal E)$
 for all $t\in\mathbb Z$. Also, $\tau(\mathcal{E})\geq 0$, and $\tau(\mathcal{E})= 0$ if and only if $\mathcal{E}$ is trivial upto line bundle twist. As examples, it is not hard to show that $\tau(T_{\mathbb{P}^n})=\frac{1}{n}$ and $\tau(\Omega_{\mathbb{P}^n})=\frac{n-1}{n}$, where $T_{\mathbb{P}^n}$ and $\Omega_{\mathbb{P}^n}$ are tangent and cotangent bundles of $\mathbb{P}^n$.

 Our result is the following.
\begin{manualtheorem}{B}\label{B}
    Let $n,l,r$ be positive integers, $X$ a nondegenerate smooth
subvariety of degree $l$ in $\mathbb{P}^n$, and
$\mathcal{E}$ a rank $r$ vector bundle on $\mathbb{P}^n$ such that
$\mathcal{E}|_X$ is $\mathcal{O}_X(a)$-Ulrich. Suppose $\tau(\mathcal{E})<1.$

Then $\mathcal{E}(1-a)$ is globally generated and one of the following
holds:

\begin{enumerate}
\item $a=n=2$, $X= \mathbb{P}^2$, $2\mid r$, and $\mathcal{E}\cong T_{\mathbb{P}^2}^{\oplus \frac r2},$

\item $X\cong \mathbb{P}^1$, $l\mid r$,
\item $a=1$, $X=\mathbb{P}^n$, and $\mathcal{E}$ is trivial.

\item $a=1$, $l=4$, $n=4$ or $5$, and $X\cong \mathbb{P}^2.$

\item $a=1$, $X\cong \mathbb{P}_{\mathbb{P}^1}(E)$
for some ample vector bundle $E$ on $\mathbb{P}^1$, and $\mathcal{O}_X(1)\simeq
\mathcal{O}_{\mathbb{P}(E)}(1).$
\end{enumerate}

Also, except in (1), we have $\tau(\mathcal{E})=\frac{l-1}{l}.$
\end{manualtheorem}
As an application to Theorem \ref{B}, we determine when twists of the wedge powers and some symmetric and tensor powers of $T_{\mathbb{P}^n}$ restricts to Ulrich bundles on subvarieties.
\begin{manualtheorem}{C}\label{C}
    Let $a,k,n,p$ be integers with $p,a,n>0$ and $p<n$.
Suppose $X$ is a nondegenerate smooth subvariety of $\mathbb{P}^n$.

\begin{enumerate}
\item Suppose $\left(\bigwedge^p T_{\mathbb{P}^n}\right)(k)\big|_X$
is $\mathcal{O}_X(a)$-Ulrich. Then $p=n-1$ and one of the following holds:

\begin{enumerate}
\item $X$ is a rational normal curve, $a=1$ and $k=-(n-1)$.

\item $a=n=2$, $X=\mathbb{P}^2$ and $k=0$.
\end{enumerate}
\item Suppose $\left(S^p T_{\mathbb{P}^n}\right)(k)\big|_X$
is $\mathcal{O}_X(a)$-Ulrich. Then the same conclusion as in (1) holds.
\item Suppose $T_{\mathbb{P}^n}^{\otimes p}(k)\big|_X$
is $\mathcal{O}_X(a)$-Ulrich. Then the same conclusion as in (1) holds.
\end{enumerate}
\end{manualtheorem}
Here by a \textit{rational normal curve} we mean the image of the $n$-uple embedding $\mathbb{P}^1\hookrightarrow \mathbb{P}^n$, or its translate by an automorphism of $\mathbb{P}^n.$

The converse of Theorem \ref{C} in each case is true, and can be proved using Lemma \ref{Tomega}. Since $\Omega^p_{\mathbb{P}^n}$ is a twist of $\bigwedge^{n-p} T_{\mathbb{P}^n}$, Theorem \ref{C} (1) also determines when $\Omega^p_{\mathbb{P}^n}(k)|_X$ is $\mathcal{O}_X(a)$-Ulrich. In particular, the cases $p=1$ and $n-1$ of Theorem \ref{C} $(1)$ is equivalent to the characterization of $(X,a,k)$ such that $T_{\mathbb{P}^n}(k)|_X$ or $\Omega_{\mathbb{P}^n}(k)|_X$ is $\mathcal{O}_X(a)$-Ulrich, which is essentially the content of \cite[Theorem 1.2]{lopez2026ulrichness}, as explained in \cite[\S 2]{lopez2026ulrichness}. So, our result generalizes \cite{lopez2026ulrichness}.

Finally, we give a whole class of examples of Ulrich bundles on curves that extend to vector bundles on the ambient space $\mathbb{P}^n.$

\begin{manualtheorem}{D}\label{D}
 \begin{enumerate}
     \item Let $X$ be a smooth projective curve that is a complete intersection in $\mathbb{P}^n$. Then there is a nontrivial Ulrich bundle on $X$ that extends to a vector bundle on $\mathbb{P}^n$.
     \item
   Let $a,l,n$ be positive integers. Suppose either
$l=n=2$ or $l\ge n\ge 3$. Let $X$ be a smooth rational
curve of degree $l$ in $\mathbb{P}^n$, and $\mathcal{F}$ an
$\mathcal{O}_X(a)$-Ulrich bundle on $X$ of rank $r$ with
$l\mid r$. Then $\mathcal{F}$ extends to a vector bundle $\widetilde{\mathcal{F}}$ on
$\mathbb{P}^n$ with $\widetilde{\mathcal{F}}(1-a)$ globally generated.  
\end{enumerate} 
\end{manualtheorem}
A couple of comments regarding the global generation conclusion of Theorem \ref{D}(2): given $\mathcal{F}$ on $X$, we want to extend it to $\widetilde{\mathcal{F}}$ on $\mathbb{P}^n$ that has as much positivity property as possible. Since $H^0(X,\mathcal{F}(-a))=0$ by Ulrichness, we cannot have $\widetilde{\mathcal{F}}(-a)$ to be globally generated. In this sense, Theorem \ref{D}(2) gives the optimal result.  Without the global generation, Theorem \ref{D}(2) would also follow from \cite{mathur2020extending}. Our new content here is the global generation, which is a very strong property for vector bundles on $\mathbb{P}^n$, see \cite{anghel2018globally}. 

\section{Extending Ulrich bundles from complete intersections}
In this section, we will prove Theorem \ref{A}. First we prove some Lemmas. Recall that a vector bundle $\mathcal{E}$ on a subvariety $X$ on $\mathbb{P}^N$ is called \textit{arithmetically Cohen-Macaulay} (ACM for short) if $H^i(X,\mathcal{E}(k))=0$ for all $k\in\mathbb{Z}$ and $0<i<\dim X$. Ulrich bundles are ACM by \cite[\S 2, (2.1)]{beauville2018introduction}.
\begin{lemma}\label{ACM}
Let $X\subset \mathbb{P}^N$ be a Cohen--Macaulay subvariety of
dimension $n\ge 3$, $H$ a hypersurface in $\mathbb{P}^N$
not containing $X$. Suppose $D:=X\cap H$
is integral, and $\mathcal{E}$ is a vector bundle on $X$ such that
$\mathcal{E}|_D$ is ACM. Then $\mathcal{E}$ is ACM.
\end{lemma}

\begin{proof}
Replacing $\mathcal{E}$ by $\mathcal{E}(k)$ for $k\in \mathbb Z$, it suffices to show $H^i(X,\mathcal{E})=0$
for $0<i<n.$ The short exact sequence
\begin{equation}\label{exact}
0\to \mathcal{E}(m)\to \mathcal{E}(m+1)\to \mathcal{E}(m+1)|_D\to 0
\end{equation}
for $m\in\mathbb Z$ and the assumption that $\mathcal{E}|_D$ is ACM
shows that there are injections
\[
H^i(X,\mathcal{E}(m))
\hookrightarrow
H^i(X,\mathcal{E}(m+1))
\text{ for }
2\le i<n
\text{ and } m\in\mathbb{Z}.
\]

Thus, we have an injection $H^i(X,\mathcal{E})
\hookrightarrow
H^i(X,\mathcal{E}(t))$
for all $t> 0$ and $2\le i<n.$ By Serre vanishing, $H^i(X,\mathcal{E}(t))=0$
for $t\gg 0,$
hence $H^i(X,\mathcal{E})=0$
for $2\le i<n.$

Again by \eqref{exact}, and using $\dim D\ge 2$, we have surjections
\[
H^1(X,\mathcal{E}(m))
\twoheadrightarrow
H^1(X,\mathcal{E}(m+1))
\]
for all $m\in\mathbb Z$. Thus we have a surjection $H^1(X,\mathcal{E}(-t))
\twoheadrightarrow
H^1(X,\mathcal{E})$
for all $t> 0$. As $X$ is Cohen-Macaulay, we have $H^1(X,\mathcal{E}(-t))=0$
for all $t\gg 0$ by
\cite[Theorem III.7.6]{hartshorne2013algebraic}.
Thus, $H^1(X,\mathcal{E})=0.$
\end{proof}
\begin{lemma}\label{easy}
Suppose $X$ is a nondegenerate smooth subvariety of $\mathbb{P}^N$ and
$a>0$, $k\in \mathbb{Z}$ are such that $\mathcal{O}_X(k)$
is $\mathcal{O}_X(a)$-Ulrich. Then $X=\mathbb{P}^N$ and $a=k+1.$
Further, if $N\ge 2$, then $k=0$.
\end{lemma}
\begin{proof}
Let $n=\dim X.$ Since $h^0\bigl(X,\mathcal{O}_X(k-a)\bigr)=0,$
we have $k\le a-1.$ Note that
\begin{equation}\label{serre}
h^0\bigl(X,K_X(na-k)\bigr)
=
h^n\bigl(X,\mathcal{O}_X(k-na)\bigr)
=
0
\end{equation}
by Serre duality.

Fujita's conjecture is true for basepoint-free ample line bundles
(see \cite[\S 5.6]{kollar1997singularities}),
So, $K_X(n+1)$
is globally generated.

This forces $na-k\le n,$
hence
\[
n(a-1)\le k\le a-1.
\]

So equality holds, hence $a=k+1,$ and $(a-1)(n-1)=0.$

If $n\ge 2$, then $a=1$, so $k=0$.
By \cite[\S 2, (2.1)]{beauville2018introduction}, $
\deg X=h^0(X,\mathcal{O}_X)=1.$ So $X=\mathbb{P}^N$, being nondegenerate.

If $n=1$, then by \eqref{serre}, $h^0\bigl(X,K_X(1)\bigr)=0,$
so $h^0(X,K_X)=0.$ This implies $X\cong \mathbb{P}^1,$
and $\deg K_X(1)<0.$ As $\deg K_X(1)=-2+\deg X,$
we get $\deg X=1,$
so $N=1$ and $X=\mathbb{P}^1$.
\end{proof}
Now we are ready to prove Theorem \ref{A}.

\textit{Proof of Theorem \ref{A}:}
\begin{proof}
    As $X$ is a complete intersection, there are Cohen--Macaulay
subvarieties of $\mathbb{P}^n$
\[
X=X_0\subsetneq X_1\subsetneq \cdots \subsetneq X_c=\mathbb{P}^n,
\]
where each $X_i$ is the intersection of $X_{i+1}$ with a hypersurface
in $\mathbb{P}^n$.

Since $\mathcal{E}|_X$ is ACM, being Ulrich, repeated application of Lemma \ref{ACM} shows that $\mathcal{E}$ is ACM.
By Horrocks' criterion (see \cite[Section 2.3]{okonek1980vector}), $\mathcal{E}$ is split, that is, $\mathcal{E}\cong \bigoplus_i \mathcal{O}_{\mathbb{P}^n}(a_i).$
So each $\mathcal{O}_X(a_i)$ is Ulrich. By Lemma \ref{easy}, we have $X=\mathbb{P}^n$ and $a_i=0$ for all $i$. So $\mathcal{E}$ is trivial.
\end{proof}
\section{Extending Ulrich bundles from arbitrary subvarieties}
In this section, we will prove Theorem \ref{B} and \ref{C}. First we prove the following Lemmas.
\begin{lemma}\label{tau}
Let $n$ be a positive integer, $\mathcal E$ a vector bundle on
$\mathbb P^n$ of rank $r>0.$ We have the following:

\begin{enumerate}
\item For $1\le p\le n-1$, we have $
\mu\!\left(\bigwedge^p\mathcal E\right)
=
p\,\mu(\mathcal E),$
and $
\tau\!\left(\bigwedge^p\mathcal E\right)
\le
p\,\tau(\mathcal E)$,

\item For $p>0$,
we have $
\mu\!\left(S^p\mathcal E\right)
=
p\,\mu(\mathcal E),$
and $
\tau\!\left(S^p\mathcal E\right)
\le
p\,\tau(\mathcal E)$,
\item For $p>0$,
we have $
\mu\!\left(\mathcal E^{\otimes p}\right)
=
p\,\mu(\mathcal E),$
and $
\tau\!\left(\mathcal E^{\otimes p}\right)
\le
p\,\tau(\mathcal E)$,

\item If $H$ is a linear subvariety of $\mathbb{P}^n$ of dimension $\geq 1$, then $\tau(\mathcal E|_H)\leq \tau(\mathcal E)$.
\end{enumerate}
\end{lemma}
\begin{proof}
  For (4), note that if $\mathcal E(k)$ is globally generated then so is $\mathcal E(k)|_H$.

To prove first part of (1), we can restrict $\mathcal{E}$ to a line. So we can assume $n=1$, hence $\mathcal E$ is split. The rest is a
standard computation and can also be deduced from
\cite[Lemma 2.1]{ghosh2026kov}.
For the second part, note that if $k$ is the smallest integer such that
$\mathcal E(k)$ is globally generated, then $\bigwedge^p(\mathcal E(k))
=
(\bigwedge^p\mathcal E)(kp)$
is also globally generated. So
\[
\tau\!\left(\bigwedge^p\mathcal E\right)
\le
\mu\!\left(\bigwedge^p(\mathcal E(k))\right)
=
p\,\mu(\mathcal E(k))
=
p\,\tau(\mathcal E).
\]

Proofs of (2) and (3) are exactly similar.
\end{proof}
\begin{lemma}\label{Tomega}
Let $n$ be a positive integer, $i:\mathbb P^1\hookrightarrow \mathbb P^n$
the $n$-uple embedding. We have
\begin{enumerate}
\item[(i)] $i^*\!\left(\Omega_{\mathbb P^n}(2)\right)
\cong
\mathcal O_{\mathbb P^1}(n-1)^{\oplus n}.$

\item[(ii)] $i^*\!\left(T_{\mathbb P^n}(-1)\right)
\cong
\mathcal O_{\mathbb P^1}(1)^{\oplus n}.$
\end{enumerate}
\end{lemma}

\begin{proof}
$(i)$: It suffices to show that $
i^*\!\left(\Omega_{\mathbb P^n}(2)\right)$
is $\mathcal{O}_{\mathbb{P}^1}(n)$-Ulrich. This can be seen by restricting the Euler exact sequence of
$\mathbb P^n$ to $\mathbb P^1$. It also follows from the ``if'' part of
\cite[Theorem 1.2(ii)]{lopez2026ulrichness}.

(ii): Follows from (i), and the observation that
\[
T_{\mathbb P^n}(-1)
\cong
\bigwedge^{n-1}\!\left(\Omega_{\mathbb P^n}(2)\right)
\otimes
\mathcal O_{\mathbb P^n}(-n+2).
\]
\end{proof}
\textit{Proof of Theorem \ref{B}:}
\begin{proof}
  Suppose the Theorem is not true. We want to get a contradiction. Let $X$ be such that the Theorem fails and $d=\dim X$ be smallest, let $L=\mathcal{O}_X(1), c=c_1(\mathcal{E}).$ So, the top self-intersection of  $L$ is $l.$ As $X$ is nondegenerate and $n>0$, we have $d>0.$
We consider two cases separately.
\noindent

\textbf{Case 1: $d=1$.}

Let $c=c_1(\mathcal{E})$, and $k$ the smallest integer so that $\mathcal{E}(k)$ is globally generated. So, $\tau(\mathcal{E})=\frac{c_1(\mathcal{E}(k))}{r}
=
\frac{c}{r}+k$.
Also, $\mathcal{E}(k')$  is globally generated for all $k'\ge k.$

Since $\mathcal{E}|_X$ is $\mathcal{O}_X(a)$-Ulrich, so
\[
H^{\bullet}\bigl(X,\mathcal{E}(-a)|_X\bigr)=0.
\]

So $\mathcal{E}(-a)$ is not globally generated, hence $-a\le k-1,$
that is,
\begin{equation}\label{a+k}
    a+k\ge 1.
\end{equation}

We have
\[
\deg(\mathcal{E}(-a)|_X)=l(c-ra),
\]
and
\[
\operatorname{rank}(\mathcal{E}(-a)|_X)=r.
\]

By Riemann-Roch,
\begin{align*}
0&=\chi\bigl(X,\mathcal{E}(-a)|_X\bigr)\\
&=
\deg(\mathcal{E}(-a)|_X)
+
(1-g)\operatorname{rank}(\mathcal{E}(-a)|_X)\\
&=l(c-ra)+r(1-g).
\end{align*}

So, $cl=r(g-1+al),$
that is,
\begin{equation}\label{c/r}
\frac{c}{r}
=
\frac{g-1}{l}+a.
\end{equation}
hence,
\begin{equation}\label{curve}
\tau(\mathcal E)
=
\frac{g-1}{l}+a+k.
\end{equation}

Now using \eqref{a+k},
\[
1>\tau(\mathcal E)
=
\frac{g-1}{l}+a+k
\ge
\frac{g-1}{l}+1,
\]
so $\frac{g-1}{l}<0,$ hence $g=0$. So, $X\cong \mathbb P^1$.

Now by \eqref{curve},
\[
1>\tau(\mathcal E)=a+k-\frac1l,
\]
so using \eqref{a+k} again,
\[
1\le a+k<1+\frac1l.
\]

Hence $a+k=1$, and $\tau(\mathcal E)=\frac{l-1}{l}.$ Also, $\mathcal E(1-a)=\mathcal E(k)$ is globally generated. Further, by \eqref{c/r}, $\frac{r}{l}=ar-c\in\mathbb{Z}$, so $r\mid l$.  So (2) of Theorem \ref{B} holds, a contradiction.

\medskip

\noindent
\textbf{Case 2: $d\ge 2$.}

\medskip

\noindent
\textbf{Subcase 1: $a>1$.}

Let $F$ be a general degree $a$ hypersurface in $\mathbb P^n$ and let $Y=F\cap X.$

Then $Y$ is nondegenerate in $\mathbb P^n$: otherwise there is a hyperplane
$H$ with $F\cap X\subseteq H\cap X,$ taking degrees we get $al\le l,$
a contradiction.

By \cite[\S 2, (2.4)]{beauville2018introduction}, $\mathcal E|_Y$ is $\mathcal O_Y(a)$-Ulrich.
By minimality of $d$, Theorem \ref{B} is valid for $(Y,\mathcal E)$.
As $\dim Y<\dim X\le n$
and $a>1$, we see that (2) of Theorem \ref{B} is the only possibility, with $(X,l)$ replaced by
$(Y,\deg Y=al)$. So, $d=2$ and $Y\cong \mathbb P^1.$

Noting that $\mathcal O_X(Y)\cong \mathcal O_X(a)$
and applying adjunction,
\[
-2=aL\cdot(aL+K_X)
=a^2L^2+aL\cdot K_X
=a\bigl(L\cdot K_X+al\bigr).
\]
Hence $a\mid 2.$ As $a>1$, we have $a=2.$ So $L\cdot K_X=-2l-1.$

Applying adjunction for a smooth hyperplane section of $X$ of genus $g\ge 0$, we get
\[
-2\le 2g-2=L\cdot(L+K_X)
=l+L\cdot K_X
=-l-1.
\]
So, $l\le 1,$ that is, $l=1.$ As $X$ is nondegenerate, we have $X=\mathbb P^n,$ and 
$n=d=2.$

As $\mathcal E$ is $\mathcal O_{\mathbb{P}^2}(2)$-Ulrich, by
\cite[\S 5]{coskun2017ulrich}, we have  $2\mid r$ and $\mathcal E\cong T_{\mathbb P^2}^{\oplus r/2}.$

Thus (1) of Theorem \ref{B} holds, a contradiction.

\medskip

\noindent
\textbf{Subcase 2: $a=1$.}

Let $H'$ be a general linear subvariety of codimension $d-1$ in
$\mathbb P^n$, $Y:=X\cap H',$ and $H$ be the linear span of $Y$.
So, $Y$ is a smooth connected curve, $Y\hookrightarrow H$ is
nondegenerate, and $\mathcal E|_Y$ is $\mathcal O_Y(1)$-Ulrich by \cite[\S 2, (2.4)]{beauville2018introduction}.
Also by Lemma \ref{tau}, we have $\tau(\mathcal E|_H)\le \tau(\mathcal E)<1.$

Now minimality of $d$ shows Theorem \ref{B} is true for
$(Y,\mathcal E|_H)$. Hence $Y\cong \mathbb P^1$. So, the sectional genus $g(X,L)$ of $X$ is $0$. By \cite[Theorem 1.2]{horing2010sectional}, we have $\Delta (X,L)=0.$ Using nondegeneracy of $X$ and \cite[Corollary 4.3 and Theorem 4.9]{fujita1982polarized}, we see that one of (3)-(5) in Theorem \ref{B} holds, or $X$ is a smooth quadric hypersurface. But this last case is ruled out by Theorem \ref{A}.

Also, as Theorem \ref{B} is true for
$(Y,\mathcal E|_H)$, we have $\mathcal E|_H$ globally generated, and $\tau(\mathcal E|_H)=\frac{l-1}{l}.$ As $\mathcal E|_Y$ is Ulrich, so $\mathcal E(-1)|_H$ is not globally generated. So, $\tau(\mathcal E|_H)=\mu(\mathcal E|_H)$. If $\mathcal{E}$ is not globally generated, then $$1>\tau(\mathcal{E})\geq \mu(\mathcal E(1))=\mu(\mathcal E|_H)+1=2-\frac{1}{l}\geq 1,$$ a contradiction. So, $\mathcal{E}$ is globally generated, $\tau(\mathcal{E})=\mu(\mathcal{E})=\mu(\mathcal E|_H)=\frac{l-1}{l}.$ So, again Theorem \ref{B} holds, a contradiction.
\end{proof}
Now we prove Theorem \ref{C}, as an application of Theorem \ref{B}.

\textit{Proof of Theorem \ref{C}:}
\begin{proof}
First of all, note that by Lemma \ref{tau}, we have $$\tau\!\left(
\left(\bigwedge^p T_{\mathbb{P}^n}\right)(k)
\right)
\le
\frac{p}{n}<1,$$ so we can apply Theorem \ref{B}. We consider two cases separately.

\medskip

\noindent
\textbf{Case 1:} $\dim X=1$.

As $X$ is nondegenerate, we have $l\ge n$. By Theorem \ref{B},
$X\cong \mathbb{P}^1$, and
\[
1-\frac1l
=
\tau\!\left(
\left(\bigwedge^p T_{\mathbb{P}^n}\right)(k)
\right)
\le
\frac{p}{n}
\le
1-\frac1n
\le
1-\frac1l.
\]

This forces $p=n-1$ and $l=n$, hence $X$ is a rational normal curve. The rest is an easy computation using Lemma \ref{Tomega}, noting that $\left(\bigwedge^{n-1}(T_{\mathbb P^n}(-1))\right)\Big|_X$
is a direct sum of $\mathcal O_{\mathbb P^1}(n-1)$.

\medskip

\noindent
\textbf{Case 2:} $\dim X\ge 2$.

If $a>1$, then we see that only possibility is that (1) in Theorem \ref{B}
holds with $r=2$. So (b) holds.

Now suppose $a=1$. Let $H$ be a general hyperplane in
$\mathbb{P}^n$ and $Y=X\cap H$. Since
\[
T_{\mathbb{P}^n}|_H
\simeq
T_H\oplus \mathcal{O}_H(1).
\]
Note that $\left(\bigwedge^p T_{\mathbb{P}^n}\right)(k)\Big|_H$
contains both $
\left(\bigwedge^p T_H\right)(k)$
and $\left(\bigwedge^{p-1}T_H\right)(k+1)$
as direct summands.

By \cite[\S 2, (2.4)]{beauville2018introduction}, $\left(\bigwedge^p T_{\mathbb{P}^n}\right)(k)\Big|_Y$
is $\mathcal{O}_Y(1)$-Ulrich, so both $\left(\bigwedge^p T_H\right)(k)\Big|_Y$ and 
$\left(\bigwedge^{p-1}T_H\right)(k+1)\Big|_Y$
are $\mathcal{O}_Y(1)$-Ulrich.

By \cite[Lemma 2.4(iii)]{casanellas2012stable},
we have
\[
\mu\!\left(\left(\bigwedge^pT_H\right)(k)\right)
=
\mu\!\left(\left(\bigwedge^{p-1}T_H\right)(k+1)\right).
\]

Now using Lemma \ref{tau}, this forces $p=p-1,$
a contradiction.

\end{proof}
\begin{remark}
    Since $\tau\bigl(\mathcal{O}_{\mathbb{P}^N}(k)\bigr)=0,$ Lemma \ref{easy} can also be considered as a consequence of Theorem \ref{B}.
\end{remark}
\section{Extending vector bundles from rational curves}
In this section, we prove Theorem \ref{D}. We need the following Lemma.
\begin{lemma}\label{algebra 2}
Let $l \geq 2$ be an integer, and
\[
k[X,Y]_{l^2-l}\otimes k[X,Y]_{l-2}
\xrightarrow{\ \beta\ }
k[X,Y]_{l^2-2}
\]
be the multiplication map. If $V \subseteq k[X,Y]_{l-2}$
is a general subspace of dimension $l+1$, then $\ker \beta \cap \bigl(V\otimes k[X,Y]_{l-2}\bigr)=0.$
\end{lemma}

\begin{proof}
It suffices to exhibit some subspace $V \subseteq k[X,Y]_{l-2}$
of dimension $l+1$ such that $\ker \beta \cap \bigl(V\otimes k[X,Y]_{l-2}\bigr)=0.$

We claim that
\[
V:=\operatorname{span}
\left\{
X^{i(l-1)}Y^{(l-i)(l-1)}
\;\middle|\;
0\le i\le l
\right\}
\]
works.

Note that
\[
\operatorname{codim}\ker\beta
=
\dim k[X,Y]_{l^2-2}
=
l^2-1
=
(l-1)(l+1)
=
\dim\bigl(V\otimes k[X,Y]_{l-2}\bigr).
\]
So it suffices to show that the restriction of $\beta$,
\[
V\otimes k[X,Y]_{l-2}
\xrightarrow{\ \beta_1\ }
k[X,Y]_{l^2-2}
\]
is surjective.

Let $0\le t\le l^2-2.$
We show $X^tY^{l^2-2-t}\in \operatorname{im}\beta_1.$
Write $t=(l-1)q+r,$ with $0\le q\le l, 0\le r\le l-2.$

We have
\[
X^tY^{l^2-2-t}
=
X^{q(l-1)}Y^{(l-q)(l-1)}
\cdot
X^rY^{l-2-r}
\in \operatorname{im}\beta_1.
\]
\end{proof}
The following Theorem is key to the proof of Theorem \ref{D}(2).
\begin{theorem}\label{key}
Let $i:\mathbb{P}^1 \hookrightarrow \mathbb{P}^n$
be an embedding whose image is not contained in a hyperplane and $i^*\mathcal{O}_{\mathbb{P}^n}(1)=\mathcal{O}_{\mathbb{P}^1}(l).$
Then there is a globally generated vector bundle $\mathcal{E}$ on $\mathbb{P}^n$ with $i^*\mathcal{E} \cong \mathcal{O}_{\mathbb{P}^1}(l-1)^{\oplus l}$.
\end{theorem}
\begin{proof}
Note that since the image of $i$ is not contained in a hyperplane, we have $l\geq n$. 

We first introduce the following notation. For integers
$n,d\ge 1$, let $P_n(d)$ be the dual of the evaluation map
\[
H^0(\mathbb{P}^n,\mathcal{O}_{\mathbb{P}^n}(d))
\otimes
\mathcal{O}_{\mathbb{P}^n}
\longrightarrow
\mathcal{O}_{\mathbb{P}^n}(d).
\]
So, $P_n(d)$ is a globally generated vector bundle on
$\mathbb{P}^n$ fitting into an exact sequence
\begin{equation}\label{syzygy}
0
\longrightarrow
\mathcal{O}_{\mathbb{P}^n}(-d)
\longrightarrow
H^0(\mathbb{P}^n,\mathcal{O}_{\mathbb{P}^n}(d))^*
\otimes
\mathcal{O}_{\mathbb{P}^n}
\longrightarrow
P_n(d)
\longrightarrow
0.
\end{equation}

For example, by the Euler exact sequence $P_n(1)=T_{\mathbb{P}^n}(-1).$

By \cite[Lemma 2.1(a)]{coskun2019normal}, $P_1(d)$ is a balanced vector bundle on $\mathbb{P}^1$,
hence
\begin{equation}\label{p1d}
P_1(d)\cong \mathcal{O}_{\mathbb{P}^1}(1)^{\oplus d}.
\end{equation}

Using \eqref{syzygy}, for $n\ge 2$, we have a natural identification
\[
H^0(\mathbb{P}^n,P_n(d))
=
H^0(\mathbb{P}^n,\mathcal{O}_{\mathbb{P}^n}(d))^*.
\]

Now we return to the proof. Note that by \cite[Theorem, (i)]{gruson1983theorem}, the map 
\[
H^0(\mathbb{P}^n,\mathcal{O}_{\mathbb{P}^n}(l-1))
\xlongrightarrow{\eta}
H^0(\mathbb{P}^1,\mathcal{O}_{\mathbb{P}^1}(l^2-l)),
\] induced by $i$, is surjective. Let $W$ be the kernel of $\eta$. We have the following commutative diagram of vector bundles on $\mathbb{P}^1$ with exact rows and columns:

\begin{equation}\label{1}
\begin{tikzcd}[
column sep=1.5em,
row sep=1.8em
]
&
&
0 \arrow[d]
&
\\
0 \arrow[r]
&
\mathcal{O}_{\mathbb{P}^1}(-(l^2-l))
\arrow[r]
\arrow[d,equal]
&
H^0(\mathbb{P}^1,\mathcal{O}_{\mathbb{P}^1}(l^2-l))^*
\otimes \mathcal{O}_{\mathbb{P}^1}
\arrow[r]
\arrow[d,"\eta^*"]
&
P_1(l^2-l)
\arrow[r]
\arrow[d,"h"]
&
0
\\
0 \arrow[r]
&
\mathcal{O}_{\mathbb{P}^1}(-(l^2-l))
\arrow[r]
&
H^0(\mathbb{P}^n,P_n(l-1))
\otimes \mathcal{O}_{\mathbb{P}^1}
\arrow[r]
\arrow[d]
&
i^*P_n(l-1)
\arrow[r]
&
0
\\
&
&
W^*\otimes\mathcal{O}_{\mathbb{P}^1}
\arrow[d]
&
\\
&
&
0
&
\end{tikzcd}
\end{equation}
Here $h$ is induced by $\eta^*$.
By snake lemma, $h$ is injective with cokernel $\cong W^*\otimes\mathcal{O}_{\mathbb{P}^1},$ so we have short exact sequence
\begin{equation}\label{2}
0
\to
P_1(l^2-l)
\xrightarrow{h}
i^*P_n(l-1)
\xrightarrow{q}
W^*\otimes\mathcal{O}_{\mathbb{P}^1}
\to 0.
\end{equation}

Since $P_1(l^2-l)$ is a globally generated vector bundle on
$\mathbb{P}^1$, we have $H^1(\mathbb{P}^1,P_1(l^2-l))=0.$
So \eqref{2} yields a short exact sequence
\[
0
\to
H^0(\mathbb{P}^1,P_1(l^2-l))
\to
H^0(\mathbb{P}^1,i^*P_n(l-1))
\xrightarrow{q_0}
W^*
\to 0.
\]

By the first and second row of \eqref{1}, we have a commutative diagram with all maps injective
\begin{equation}\label{3}
\begin{array}{ccc}
H^0(\mathbb{P}^1,\mathcal{O}_{\mathbb{P}^1}(l^2-l))^*
&\rightarrow&
H^0(\mathbb{P}^1,P_1(l^2-l))
\\
\downarrow && \downarrow
\\
H^0(\mathbb{P}^n,P_n(l-1))
&\longrightarrow&
H^0(\mathbb{P}^1,i^*P_n(l-1)).
\end{array}
\end{equation}

By this, we will consider the vector spaces in \eqref{3} as subspaces of $H^0(\mathbb{P}^1, i^*P_n(l-1)).$

Note that
\[
q_0\big|_{H^0(\mathbb{P}^n,P_n(l-1))}
:
H^0(\mathbb{P}^n,P_n(l-1))
\to
W^*
\]
is same as the map in $H^0$ induced by the last map in the second column of \eqref{1}, hence it is surjective with kernel $=H^0(\mathbb{P}^1,\mathcal{O}_{\mathbb{P}^1}(l^2-l))^*.$

So, \eqref{3} is a fibre square and 
\[
H^0(\mathbb{P}^n,P_n(l-1))
+
H^0(\mathbb{P}^1,P_1(l^2-l))
=
H^0(\mathbb{P}^1,i^*P_n(l-1)).
\]

Let $a:=l^2-2l+\dim W,$
and let $V\subseteq H^0(\mathbb{P}^n,P_n(l-1))$ a general $a$-dimensional subspace. As $a\ge \dim W,$
we have $q_0(V)=W^*.$ Let $V_1
=V\cap
\ker q_0,$
a general subspace of $H^0(\mathbb{P}^1,\mathcal{O}_{\mathbb{P}^1}(l^2-l))^*$
of dimension $l^2-2l.$ Let $V_2
\subseteq
H^0(\mathbb{P}^1,\mathcal{O}_{\mathbb{P}^1}(l^2-l))
=
k[X,Y]_{l^2-l}
$
be the annihilator of $V_1$.

As
\begin{align*}
a
&=
l^2-2l
+
h^0(\mathbb{P}^n,\mathcal{O}_{\mathbb{P}^n}(l-1))
-
h^0(\mathbb{P}^1,\mathcal{O}_{\mathbb{P}^1}(l^2-l))\\
&=
h^0(\mathbb{P}^n,\mathcal{O}_{\mathbb{P}^n}(l-1))
-l-1\\
&\le
h^0(\mathbb{P}^n,\mathcal{O}_{\mathbb{P}^n}(l-1))
-n-1\\
&=
\operatorname{rank}P_n(l-1)-n,
\end{align*}
$V$ induces a short exact sequence of vector bundles on
$\mathbb{P}^n$:
\begin{equation}\label{11}
0
\to
V\otimes\mathcal{O}_{\mathbb{P}^n}
\to
P_n(l-1)
\to
\mathcal{E}
\to
0.
\end{equation}

So $\mathcal{E}$ is a globally generated vector bundle of rank $l$ on
$\mathbb{P}^n$. We shall show that $i^*\mathcal{E}\cong \mathcal{O}_{\mathbb{P}^1}(l-1)^{\oplus l}.$

It suffices to show that $i^*\mathcal{E}$
is $\mathcal{O}_{\mathbb{P}^1}(l)$-Ulrich, that is, $H^\bullet\!\left(\mathbb{P}^1,(i^*\mathcal{E})(-l)\right)=0.$

As $V_1
\subseteq
H^0(\mathbb{P}^1,P_1(l^2-l)),$
we have an induced map
\[
V_1\otimes\mathcal{O}_{\mathbb{P}^1}
\xrightarrow{\,j'\,}
P_1(l^2-l).
\]

We have the following commutative diagram with exact rows:
\[
\begin{tikzcd}
0 \arrow[r]
&
V_1\otimes\mathcal{O}_{\mathbb{P}^1}
\arrow[r]
\arrow[d,"j'"]
&
V\otimes\mathcal{O}_{\mathbb{P}^1}
\arrow[r,""]
\arrow[d]
&
W^*\otimes\mathcal{O}_{\mathbb{P}^1}
\arrow[r]
\arrow[d,equal]
&
0
\\
0 \arrow[r]
&
P_1(l^2-l)
\arrow[r,"h"]
&
i^*P_n(l-1)
\arrow[r,"q"]
&
W^*\otimes\mathcal{O}_{\mathbb{P}^1}
\arrow[r]
&
0.
\end{tikzcd}
\]

The middle vertical map is the restriction of the injection in \eqref{11} to
$\mathbb{P}^1$ via $i$, hence is injective. So, snake lemma shows $i^*\mathcal{E}\cong \operatorname{coker}(j')$ and $j'$ is injective. So, we get a short exact sequence
\[
0
\to
V_1\otimes\mathcal{O}_{\mathbb{P}^1}(-l)
\xrightarrow{j}
P_1(l^2-l)(-l)
\to
(i^*\mathcal{E})(-l)
\to
0.
\]

It suffices to show $j'$ induces an isomorphism on each $H^i$.

We have a short exact sequence
\[
0
\to
P_1(l^2-l)^*(l-2)
\to
H^0(\mathbb{P}^1,\mathcal{O}_{\mathbb{P}^1}(l^2-l))
\otimes
\mathcal{O}_{\mathbb{P}^1}(l-2)
\to
\mathcal{O}_{\mathbb{P}^1}(l^2-2)
\to
0.
\]

This gives a left exact sequence
\[
0
\to
H^0(\mathbb{P}^1,P_1(l^2-l)^*(l-2))
\to
k[X,Y]_{l^2-l}\otimes k[X,Y]_{l-2}
\xrightarrow{\ \beta\ }
k[X,Y]_{l^2-2}.
\]

By Lemma \ref{algebra 2}, we get
\[
H^0(\mathbb{P}^1,P_1(l^2-l)^*(l-2))
\cap
\left(
V_2\otimes k[X,Y]_{l-2}
\right)
=0
\]
in $k[X,Y]_{l^2-l}\otimes k[X,Y]_{l-2}.$ As $V_2\otimes k[X,Y]_{l-2}$
is the kernel of the surjection
\[
k[X,Y]_{l^2-l}\otimes k[X,Y]_{l-2}
\to
V_1^*\otimes k[X,Y]_{l-2},
\] the map
\[
H^0(\mathbb{P}^1,P_1(l^2-l)^*(l-2))\to V_1^*\otimes k[X,Y]_{l-2}=
H^0\!\left(
\mathbb{P}^1,
V_1^*\otimes\mathcal{O}_{\mathbb{P}^1}(l-2)
\right)
\]

induced by $j^*$ is injective.

By Serre duality, the map on $H^1$ induced by $j$ is surjective.

By \eqref{p1d},
\begin{align*}
h^1\!\left(
\mathbb{P}^1,
P_1(l^2-l)(-l)
\right)
&=
h^1\!\left(
\mathbb{P}^1,
\mathcal{O}_{\mathbb{P}^1}(1-l)
^{l^2-l}\right)\\
&=
(l^2-l)(l-2)\\
&=
(l-1)\dim V_1\\
&=
h^1\!\left(
\mathbb{P}^1,
V_1\otimes\mathcal{O}_{\mathbb{P}^1}(-l)
\right).
\end{align*}

So, the map on $H^1$ induced by $j$ is an isomorphism.

By \eqref{p1d},
\[
h^0\!\left(
\mathbb{P}^1,
P_1(l^2-l)(-l)
\right)
=
h^0\!\left(
\mathbb{P}^1,
\mathcal{O}_{\mathbb{P}^1}(l-1)
\right)^{l^2-l}
=0
=
h^0\!\left(
\mathbb{P}^1,
V_1\otimes\mathcal{O}_{\mathbb{P}^1}(-l)
\right).
\]
So the map on $H^0$ induced by $j$ is also an isomorphism.
\end{proof}

Now we are ready to prove Theorem \ref{D}.

\textit{Proof of Theorem \ref{D}:}
\begin{proof}
\textit{(1):} Let $g$ be the genus of $X$. Choose a general line bundle on $X$ of degree $g-1$, so $\omega_X\otimes L^{-1}$ is also a general line bundle of degree $g-1$, where $\omega_X$ is the canonical line bundle of $X$. Hence $$h^0(X,L)=h^0(X,\omega_X\otimes L^{-1})=0.$$ So, by Serre duality, $$H^\bullet(X,L)=0=H^\bullet(X,\omega_X\otimes L^{-1})=0.$$ Now let $\mathcal{F}=L\oplus(\omega_X\otimes L^{-1}).$ So, $H^\bullet(X,\mathcal{F})=0$, hence $\mathcal{F}(1)$ is Ulrich. Further, $\textrm{det }\mathcal{F}\cong \omega_X$ extends to a line bundle on $\mathbb{P}^n$, as $X$ is a complete intersection. Now by \cite{mathur2020extending}, we see that the Ulrich bundle $\mathcal{F}(1)^{\oplus r}$ extends to a vector bundle on $\mathbb{P}^n$ for any integer $r\geq n/2.$

\textit{(2):}
Let $i:\mathbb{P}^1\hookrightarrow \mathbb{P}^n$
be an embedding with image $X$.

By Theorem \ref{key}, there is a globally generated vector bundle
$\mathcal{E}$ on $\mathbb{P}^n$ with
\[
i^*\mathcal{E}\cong \mathcal{O}_{\mathbb{P}^1}(l-1)^{\oplus l}.
\]

Let $\widetilde{\mathcal{F}}:=\mathcal{E}(a-1)^{\oplus \frac rl}.$
We have $i^*\mathcal{O}_X(a)
\cong
\mathcal{O}_{\mathbb{P}^1}(al),$
and the only $\mathcal{O}_{\mathbb{P}^1}(al)$-Ulrich bundle on
$\mathbb{P}^1$ of rank $r$ is $\mathcal{O}_{\mathbb{P}^1}(al-1)^{\oplus r}\cong i^*\widetilde{\mathcal{F}}$. So, $\widetilde{\mathcal{F}}|_X$
is the unique $\mathcal{O}_X(a)$-Ulrich bundle of rank $r$,
hence $\widetilde{\mathcal{F}}|_X\cong \mathcal{F}.$

Finally, as $\mathcal{E}$ is globally generated, so $\widetilde{\mathcal{F}}(1-a)$
is globally generated. 
\end{proof}
\begin{remark}
    In the notations of the proof of Theorem \ref{D}(2), note that
\[
\tau(\widetilde{\mathcal{F}})
\le
\mu(\widetilde{\mathcal{F}}(1-a))
=
\mu(\mathcal{E}^{\oplus{\frac{r}{l}}})
=
\mu(\mathcal{E})
=
\frac{l-1}{l}
<1.
\]
Hence, case (2) of Theorem \ref{B} indeed occurs. Clearly, cases (1) and (3) also occurs. We do not know whether cases (4) or (5) can occur.
\end{remark}

\section{Acknowledgement}
 I thank Prof. János Kollár and Iustin Coanda and insightful discussions. I also thank Anindya Mukherjee for making me interested in Ulrich bundles.
\printbibliography
\begin{flushleft}
{\scshape Department of Mathematics, Fine Hall, Princeton University, Princeton, NJ 700108, USA}.
{\fontfamily{cmtt}\selectfont
\textit{Email address: ss6663@princeton.edu} }
\end{flushleft}

\end{document}